\documentclass{amsart}

\input{amssym.def}
\input{amssym.tex}

\usepackage[curve,tips]{xypic}
\SelectTips{eu}{11}
\UseTips

\newtheorem{thmA}{Theorem}

\usepackage{amsthm}
\usepackage{epsfig}
\usepackage{color}

\title[Profinite rigidity, fibering, and the  figure-eight knot] 
{Profinite rigidity, fibering, and the  figure-eight knot} 
\author{M. R. Bridson}
\author{A. W. Reid}
\address{\newline Mathematical Institute, 
\newline Andrew Wiles Building,
\newline University of Oxford,
\newline Oxford OX2 6GG, UK}
\email{ bridson@maths.ox.ac.uk}
\address{\newline Department of Mathematics,
\newline University of Texas, 
\newline Austin, TX 78712, USA}
\email{ areid@math.utexas.edu}
\thanks{Bridson was supported in part by grants from the
EPSRC and a Royal Society Wolfson Merit Award. Reid was supported in part 
by an NSF grant}

\def\-{\overline}

\def\wh{\widehat}

\def\G{\Gamma}
\def\Ge{\Pi}
\def\ssm{\smallsetminus}
\def\K{{\mathcal{K}}}
\def\H{\Bbb H}
\def\Z{\Bbb Z}
\def\Q{\Bbb Q}
\def\R{\Bbb R}
\def\F{\Bbb F}
\def\Sh{\Bbb S}

\def\rk{{\rm{rk }\, }}
\def\ker{\rm{ker}}

\def\cd{{\rm{cd}}}

\def\PSL{\rm{PSL}}
\def\SL{\rm{SL}}

\def\GL{\rm{GL}}

\def\qed{ $\sqcup\!\!\!\!\sqcap$}

\def\tr{\mbox{\rm{tr}}\, }

\def\G{\Gamma}
\def\C{\mathcal{C}}
 
\def\<{\langle}
\def\>{\rangle}
\def\ilim{\varprojlim}

\def\ns{\vartriangleleft} 
\def\Gp{\Gamma_{\phi}}
\def\aut{{\rm{Aut}}}
\def\out{{\rm{Out}}}

\newtheorem{theorem}{Theorem}[section]
\newtheorem{lemma}[theorem]{Lemma}
\newtheorem{corollary}[theorem]{Corollary}
\newtheorem{proposition}[theorem]{Proposition}

\theoremstyle{definition} 

\newtheorem{remark}[theorem]{Remark}
\newtheorem{example}[theorem]{Example}

\begin{document}

\begin{abstract} We establish results concerning the
 profinite completions of 3-manifold groups. In particular, we prove that the complement of the
 figure-eight knot $\Sh^3\ssm\K$ is distinguished 
from all other compact 3-manifolds by the set of 
 finite quotients of its fundamental group. 
In addition, we show that if $M$ is a compact 3-manifold with $b_1(M)=1$,   
and $\pi_1(M)$ has the
 same finite quotients as a free-by-cyclic group $F_r\rtimes\Z$, then  $M$ has non-empty boundary,
fibres over the circle with compact fibre, 
and  $\pi_1(M)\cong F_r\rtimes_\psi\Z$ for some  
$\psi\in{\rm{Out}}(F_r)$.
\end{abstract}



\maketitle

%
%
%
%

\section{Introduction}
We are interested in the extent to which  
the set of finite quotients of
a 3-manifold group $\Gamma=\pi_1(M)$ determines $\Gamma$ and $M$. 
In this article we shall prove that fibred 3-manifolds that have non-empty boundary and first betti number $1$
can be distinguished from
all other compact 3-manifolds by examining the finite quotients of $\pi_1(M)$. 
A case where our methods are particularly well suited
is when $M$ is the complement of the figure-eight knot. Throughout,
we denote this knot by
$\K$ and write
$\Ge=\pi_1(\Sh^3\smallsetminus \K)$. 

The manifold $\Sh^3\smallsetminus \K$
and the group $\Ge$ hold a special
place in the interplay of 3-manifold topology and hyperbolic geometry.
$\K$ became the first example of a {\em hyperbolic knot} when
Riley \cite{Ri} constructed a discrete faithful 
representation of $\Ge$ into
$\PSL(2,\Z[\omega])$,  where $\omega$ is a cube root of unity.
Subsequently, Thurston \cite{Th}  showed that the hyperbolic
structure on $\Sh^3\smallsetminus \K$ could be obtained by gluing two copies of the 
regular ideal tetrahedron from ${\H}^3$.  The insights gained from understanding the complement of the
figure-eight knot provided crucial
underpinning for many of Thurston's great discoveries concerning the geometry and topology of 3-manifolds,
such as the nature of geometric structures on manifolds obtained by Dehn surgery, and the question of when a surface bundle admits a hyperbolic structure.
In our exploration of the finite quotients of 3-manifold groups,  the figure-eight knot stands out once again as a special example.

Throughout this article, we allow manifolds to have non-empty boundary, but we assume that any spherical
boundary components have been removed by capping-off with a $3$-ball. 
We write
$\C(G)$ for the set of isomorphism classes of the finite quotients of a group $G$. 

\begin{thmA}\label{main}
Let $M$ be a compact connected 3-manifold. If $\C(\pi_1(M))=\C(\Ge)$, then $M$
is homeomorphic to $\Sh^3\ssm\K$.
\end{thmA}

\begin{corollary}\label{knots}
Let $J\subset \Sh^3$ be a knot and let $\Lambda=\pi_1(S^3\smallsetminus J)$. If 
$\C(\Lambda)=\C(\Ge)$, then $J$ is isotopic to   $\K$.
\end{corollary} 

It is useful to arrange
the finite quotients of a group  $\G$ into a directed system, and to replace $\C(\G)$ by the {\emph{profinite
completion} $\wh{\G}$:
the normal subgroups 
of finite index $N\ns\G$ are ordered by reverse inclusion and  
$
\wh{\G} := \ilim \G/N.$ 
A standard argument shows that for finitely
generated groups, $\C(\G_1)=\C(\G_2)$ if and only if $\wh{\G}_1\cong \wh{\G}_2$
 (see \cite{DFPR}). Henceforth, we shall express our results in terms of $\wh{\G}$ rather
 than $\C(\G)$. 

$\Sh^3\smallsetminus \K$ has the structure of a once-punctured torus
bundle over the circle (see \S 3), and hence $\Ge$ is an extension
of a free group of rank $2$ by $\Z$. We shall use Theorem \ref{main} to frame a broader investigation
into 3-manifolds whose fundamental groups have the same profinite completion as a  free-by-cyclic group.
To that end, we fix some notation.

We write $F_r$ to denote the free group of rank $r\geq 2$ 
and $b_1(X)$ to denote the first betti number $\rk {\rm{H}}^1(X,\Z)$,
where $X$ is a space or a finitely generated group. Note that for knot complements $b_1(X)=1$. 
Stallings' Fibering Theorem \cite{stall} 
shows that a compact 3-manifold with non-empty boundary 
fibres over the circle, with compact fibre,
if and only if $\pi_1N$ is of the form $F_r\rtimes\Z$. 

In Section \ref{s:fibres} we shall elucidate the nature of compact
3-manifolds whose fundamental groups have the same profinite
completion as a group of the form $F_r\rtimes\Z$. In particular we
shall prove:

\begin{thmA}\label{t:fibres} Let $M$ be a compact connected 3-manifold and let  
$\G=F_r\rtimes_\phi\Z$.
If $b_1(M)=1$ and $\wh{\G}\cong\wh{\pi_1(M)}$, then $M$ has non-empty boundary,
fibres over the circle with compact fibre,  
and  $\pi_1(M)\cong F_r\rtimes_\psi\Z$  for some $\psi\in{\rm{Out}}(F_r)$.
\end{thmA}

\begin{corollary}
\label{profinite_invt_fibered}
Let $M$ and $N$ be compact connected 3-manifolds with $\wh{\pi_1(N)}\cong\wh{\pi_1(M)}$.
If $M$ has non-empty incompressible boundary and fibres over the circle, and $b_1(M)=1$, then $N$ 
has non-empty incompressible
boundary and fibres over the circle, and $b_1(N)=1$.\end{corollary}

One can remove the hypothesis $b_1(M)=1$  from Theorem \ref{t:fibres} at the expense of demanding more from the
isomorphism $\wh{\pi_1M}\to\wh{F_r\rtimes\Z}$: it is enough to require that $\pi_1M$ has cyclic image under the composition
$$
\pi_1M\to\wh{\pi_1M}\to\wh{F_r\rtimes\Z}\to\wh{\Z},
$$
or else that the isomorphism $\wh{\pi_1M}\to\wh{F_r\rtimes\Z}$ is regular in the sense of \cite{BF}; see remark \ref{b_1need}. 
Corollary \ref{profinite_invt_fibered} can be strengthened in the same manner.

We shall pay particular attention to the case $r=2$. In that setting
we prove that for each $\G$ with
$b_1(\G)=1$,   there are only finitely many possibilities for $M$,
and either they are all hyperbolic or else none of them are.
(We refer the reader to Theorem \ref{t:compile} at the end of the paper
for  a compilation of related results.) These
considerations apply to the complement of the figure-eight knot,
because it is a once-punctured torus bundle with holonomy
$$
\phi =  \begin{pmatrix} 2& 1\cr 1& 1\cr\end{pmatrix} , 
$$
whence $\Ge\cong F_2\rtimes_\phi\Z$. 
In this case, a refinement of the argument which shows that there are only finitely many
possibilities for $M$ shows that in fact there is a unique possibility. 
Similar arguments show that the  trefoil knot complement and the Gieseking manifold are
uniquely defined by the finite quotients of their fundamental groups 
(see \S 3 and \S 6 for details).

The most difficult step in the proof of Theorem \ref{t:fibres} is achieved by appealing to the 
cohomological criterion provided by the following theorem. 

\begin{thmA}\label{tt:freedmen}
Let $M$ be a compact, irreducible, 3-manifold with non-empty boundary 
which is a non-empty union of incompressible tori and Klein bottles. 
Let $f:\pi_1(M)\to\Z$ be an epimorphism. 
Then, either $\ker~f$ is finitely generated and free (and $M$ is fibred) or else
the closure of $\ker~f$ in $\wh{\pi_1M}$ has cohomological dimension at least $2$.
\end{thmA}

Our proof of this theorem  relies on the breakthroughs of Agol \cite{Ag} and Wise \cite{Wi} 
concerning cubical
structures on 3-manifolds and the subsequent work of Przytycki-Wise \cite{PW} 
on the separability of embedded surface groups, as refined by Liu \cite{Liu}. We isolate from this body of work 
a technical result of independent interest: {\em{if $M$ is a closed 3-manifold and $S\subset M$
is a closed embedded incompressible surface, then the closure of $\pi_1S$ in $\wh{\pi_1M}$ is isomorphic to $\wh{\pi_1 S}$.}}
The surfaces to which we apply this result are 
obtained using a construction 
of Freedman and Freedman \cite{FF}.
Our argument also  exploits the calculation of $L_2$-betti numbers
by Lott and L\"uck \cite{LL}, and the Cyclic Surgery Theorem \cite{CGLS}. Results concerning the goodness and
cohomological dimension of profinite completions also play an important role in the proof of Theorem B.

Our results here, and similar results in \cite{BF},  are the first to give credence to the
possibility, raised in \cite{Ag2} and \cite{Re}, that Kleinian groups of 
finite co-volume might be distinguished from each other and from
all other 3-manifold groups by their profinite completions, i.e.~that they are {\em profinitely rigid}
in the sense of \cite{Re}.
In contrast, the fundamental groups of 3-manifolds modelled on the geometries {\rm{Sol}}
and $\H^2\times\R$ are not profinitely rigid in general; 
see \cite{Fu}, \cite{He2} and  Remark \ref{r:torus}. 
There are also lattices in higher-dimensional semi-simple Lie groups that
are not determined by their profinite completions 
(see \cite{Ak}).
 We refer the reader to \cite{Re} \S  10 for a wider
discussion of profinite rigidity for 3-manifold groups, and note that 
the recent work of Wilton and Zalesskii \cite{WZ} provides an important advance in this area.

Theorem \ref{t:compile} prompts the question: 
{\em to what extent are free-by-cyclic groups profinitely rigid?}
We shall return to this question in a future article (cf.~Remark \ref{r:torus}). 
Note that one cannot hope to distinguish a free-by-cyclic group from an arbitrary finitely presented,
residually finite group by means of its finite quotients without first resolving the deep question of whether finitely generated free groups themselves
can be distinguished.\\[\baselineskip]
\noindent{\bf Remark:}~During the writing of this paper we became aware
that M. Boileau and S. Friedl were working on similar problems \cite{BF}.  Some of their results
overlap with ours, but the methods of proof are very different.\\[\baselineskip]
\noindent{\bf Acknowledgement:}~We thank Yi Liu for helpful
conversations and correspondence, and Cameron Gordon for Dehn surgery
advice. We would also like to thank Stefan Friedl for helpful comments
concerning the issue in Remark 5.2.

\section{Preliminaries concerning profinite groups}

In this section we gather the results about profinite groups that we shall need.

\subsection{Basics} Let $\G$ be a finitely generated group. We have
$$
\wh{\G} = \ilim \G/N
$$
where the inverse limit is taken over the normal
subgroups of finite index $N\ns\G$ ordered
by reverse inclusion. $\wh{\G}$ is a compact topological group.

 The natural  homomorphism $i:\G\to\widehat\G$ is injective if and only if
$\G$ is residually finite.
The image of $i$  is dense regardless of whether $\G$ is residually
finite or not, so the restriction to $\G$ of any continuous epimomorphism from $\widehat{\Gamma}$
to a  finite group is onto.   A deep theorem of Nikolov and Segal
\cite{NS} implies that if $\G$ is finitely generated then
{\em every} homomorphism from $\widehat{\G}$ to a finite group is 
continuous.

 For every finite group $Q$, the restriction $f\mapsto f\circ i$, gives
a bijection from the set of epimorphisms $\wh{\G}\twoheadrightarrow Q$ to the set of
epimorphisms ${\G}\twoheadrightarrow Q$. From consideration of the finite abelian quotients, the
 following useful lemma is easily deduced.

\begin{lemma}\label{l:abel}
Let $\G_1$ and $\G_2$ be finitely generated groups. If $\wh\G_1\cong\wh\G_2$, then $H_1(\G_1,\Z)\cong H_1(\G_2,\Z)$.
\end{lemma}

\medskip
\noindent{\bf{Notation and terminology.}}  Given a subset $X$ of a profinite group
$G$, we write $\overline X$ to denote the closure of $X$ in $G$.

\smallskip

Let $\G$ be a finitely generated, residually finite group and let $\Delta$ be a subgroup.
The inclusion $\Delta\hookrightarrow\G$ induces a continuous homomorphism
$\wh{\Delta}\to\wh{\Gamma}$ whose image is $\-{\Delta}$. The map $\wh{\Delta}\to\wh{\Gamma}$
is injective if and only if $\-{\Delta}\cong\wh{\Delta}$; and in this circumstance one says that
{\em $\Gamma$ induces the full profinite topology on $\Delta$}. For finitely generated groups,
this is equivalent to the statement that for every subgroup of finite index $I<\Delta$ there is a 
subgroup of finite index $S<\G$ such that $S\cap\Delta \subset I$.

Note that if $\Delta<\G$ is of finite index, then   $\G$ induces the full
profinite topology on $\Delta$. 

\smallskip

There are two other situations of interest to us where
 $\G$ induces the full profinite topology on a 
subgroup.
First, suppose that $\G$ is a group and $H$ a subgroup
of $\G$, then $\G$ is called $H$-separable if for every $g\in
G\smallsetminus H$, there is a subgroup $K$ of finite index in $\G$
such that $H \subset K$ but $g\notin K$; equivalently, the intersection
of all finite index subgroups in $\Gamma$ containing $H$ is precisely
$H$.  The group $\G$ is called {\em LERF} (or {\em subgroup
  separable}) if it is $H$-separable for every finitely-generated
subgroup $H$, or equivalently, if every finitely generated subgroup is
a closed subset in the profinite topology.

It is important to note that even if a finitely generated subgroup
$H<\G$ is separable, it need not be the case that $\Gamma$ induces the
full profinite topology on $H$.  However, LERF does guarantee this:
each subgroup of finite index $H_0<H$ is finitely generated, hence
closed in $\G$, so we can write $H$ as a union of cosets $h_1H_0 \cup \dots
\cup h_nH_0$ and find subgroups of finite index $K_i<\G$ such that
$H_0\subset K_i$, but $h_i\not\in K_i$; the intersection of the $K_i$
is then a subgroup of finite index in $\G$ with $K\cap H= H_0$.\\[\baselineskip]

\noindent The second situation that is important for us is the following.

\begin{lemma}\label{l:full} If $N$ is finitely generated, then any semidirect product of the form $\G=N\rtimes Q$ induces the full profinite topology on $N$.
\end{lemma}

\begin{proof} The characteristic subgroups $C<N$ of finite index are a fundamental 
system of neighbourhoods of $1$ defining the profinite topology on $N$,
so it is enough to construct a subgroup of finite index $S<\G$ such that $S\cap N=C$. As $C$ is characteristic in $N$, it is invariant under the
action of $Q$, so the action $\Psi:Q\to{\rm{Aut}}(N)$ implicit in the semidirect product descends to an action
$Q\to {\rm{Aut}}(N/C)$, with image $Q_C$ say. Define $S$ to be the kernel of $\G\to (N/C)\rtimes Q_C$.
\end{proof}

\subsection{On the cohomology of profinite groups} We recall some facts about the cohomology of profinite
groups that we shall use.  
We refer the reader to \cite[Chapter 6]{RZ} and \cite{Se} for details.

Let $G$ be a profinite group,  $M$ a discrete $G$-module (i.e. an abelian group $M$ equipped with the discrete topology on which $G$ acts continuously) and  let $C^n(G,M)$ be the set of all continuous  maps $G^n\rightarrow M$. 
When equipped with the usual coboundary operator $d:C^n(G,M)\rightarrow C^{n+1}(G,M)$,
this defines a chain complex $C^*(G,M)$ whose cohomology groups 
$H^q(G;M)$ are the
{\em continuous cohomology groups} of $G$ with coefficients in $M$.

Now let $\G$ be a finitely generated group.
Following Serre \cite{Se}, we say that a group $\G$ is {\em good} 
if for all $q\geq 0$ and for every finite $\G$-module $M$, the homomorphism of cohomology groups
$$H^q(\widehat{\G};M)\rightarrow H^q(\G;M)$$
induced by the natural map $\G\rightarrow \widehat{\G}$ is an isomorphism
between the cohomology of $\G$ and the continuous cohomology of 
$\widehat{\G}$. 

Returning to the general setting again, let $G$ be a profinite group, 
the {\em $p$-cohomological dimension} of $G$ is the least integer $n$ such 
that for every discrete  $G$-module $M$ and for 
every $q>n$, the $p$-primary component of $H^q(G;M)$ is zero. This
is  denoted by $\cd_p(G)$. 
The cohomological dimension of $G$ is defined 
as the supremum of $\cd_p(G)$ over all primes $p$, 
and this is denoted by $\cd(G)$. 

We also retain the standard notation
$\cd(\G)$ for the cohomological dimension (over ${\Z}$) of a 
discrete group $\G$. 

\begin{lemma} 
\label{cdim}\label{l:cd}
If the discrete group $\G$  is good, then  $\cd(\widehat{\Gamma})\le \cd(\G)$.
If, in addition, $\G$ is the fundamental group of a closed aspherical manifold, then $\cd(\widehat{\Gamma})= \cd(\G)$.
\end{lemma}

\begin{proof} If $\cd(\G)\le n$ then $H^q(\G,M)=0$ for every $\G$-module $M$ and every $q>n$. By goodness,
 this vanishing transfers to the profinite setting in the context of finite modules.
 
If $\G$ is the fundamental group of a closed aspherical $d$-manifold $M$, then $\cd (\G) = d$. And by
Poincar\'e duality, $H^d(\G, \F_2)\neq 0$, whence $H^d(\wh{\G}, \F_2)\neq 0$.
\end{proof}

We also recall the following result (see \cite[Chapter 1 \S  3.3]{Se}).

\begin{proposition}
\label{cdbound}
Let $p$ be a prime, let $G$ be a profinite group, and $H$ a closed subgroup of $G$. 
Then $\cd_p(H) \leq \cd_p(G)$.
\end{proposition} 

Fundamental groups of closed surfaces are good and (excluding $\Sh^2$ 
and $\R {\rm{P}}^2$) have cohomological 
dimension $2$.   Free groups are good and have cohomological  dimension $1$.

\begin{corollary} 
\label{freenotsurface}
If $F$ is a non-abelian free group and $\Sigma$ is the fundamental group of a closed surface other than $\Sh^2$, then 
$\widehat F$ does not contain $\wh{\Sigma}$. 
\end{corollary}

If one has a short exact sequence $1\to N\to \G\to Q\to 1$ where
both $N$ and $Q$ are good, and  $H^q(N;M)$ is finite for every finite $\G$-module
$M$ then $\G$ is good (see  \cite[Chapter 1 \S 2.6]{Se}). Hence we have:

\begin{lemma} 
\label{l:extend}
If $F\neq 1$ is finitely generated and free, then every group of the form $F\rtimes\Z$ is good and $\cd(\wh{F\rtimes\Z})=2$.
\end{lemma}

\subsection{$L^2$-betti numbers} The standard reference for this material is L\"uck's treatise \cite{Lu}.
$L^2$-betti numbers of groups are defined analytically, but L\"uck's Approximation Theorem  provides a
purely algebraic surrogate definition of the first $L^2$-betti number
for finitely presented, residually finite groups: if $(N_i)$ is a sequence of
finite-index normal subgroups in $\G$ with $\cap_iN_i=1$, then
$b_1^{(2)}(\G) = \lim_i b_1(N_i)/[\G:N_i]$. The fundamental group of every compact 3-manifold $X$ is
residually finite (see for example
\cite{AFW}), so we can use the surrogate definition  
$b_1^{(2)}(X)=b_1^{(2)}(\pi_1(X))$.

We shall require the following result  
\cite[Theorem 3.2]{BCR}.

\begin{proposition} 
\label{dense}
Let $\Lambda$ and $\G$ be finitely presented residually
finite groups
and suppose that $\Lambda$ is a dense subgroup of  $\widehat{\Gamma}$. 
Then  $b_1^{(2)}(\Gamma)\le  b_1^{(2)}(\Lambda)$.

In particular, if $\widehat{\Lambda}\cong \widehat{\G}$ then 
$b_1^{(2)}(\Gamma) = b_1^{(2)}(\Lambda)$.
\end{proposition} 

The following special case of a result of Gaboriau \cite{gab} will also be useful.

\begin{proposition} 
\label{p:gab}
Let $1\to N\to\G\to Q\to 1$ be a short exact sequence of groups. If $\G$ is finitely presented, $N$ finitely
generated and $Q$ is infinite, then $b_1^{(2)}(\Gamma)=0$.
\end{proposition}

\section{Groups of the form $F\rtimes\Z$}\label{s:fibres}

In this section we explore the extent to which free-by-cyclic groups are determined by their profinite completions,
paying particular attention to the figure-eight knot.
We remind the reader of our notation: $b_1(G)$ is the rank of $H^1(G,\Z)$
and $F_r$ is a free group of rank $r$.
\smallskip

Let $\G$ be a finitely generated group. If $b_1(\G)=1$, then there is
a unique normal subgroup $N \subset \G$ such that $\G/N\cong
\Z$  and the kernel of the induced map $\wh{\G}\to\wh{\Z}$ is the closure
$\-{N}$ of $N$.  We saw in Lemma \ref{l:full} that if $N$ is finitely
generated then $\-{N}\cong\wh{N}$, in which case we have a short exact
sequence:

\begin{equation}\tag{$\dagger$}\label{equ}
1\to \wh{N}\to \wh{\G}\to\wh{\Z}\to 1.
\end{equation}

\begin{lemma}\label{l:sameHat} Let $\G_1=N_1\rtimes\Z$ and $\G_2=N_2\rtimes\Z$, with $N_1$ and $N_2$ finitely generated. If
$b_1(\G_1)=1$ and $\wh{\G}_1\cong \wh{\G}_2$, then $\wh{N}_1\cong\wh{N}_2$.
\end{lemma}

\begin{proof} For finitely generated groups, $\wh{\G}_1\cong \wh{\G}_2$ implies $H_1(\G_1,\Z)\cong H_1(\G_2,\Z)$.
Thus $b_1(\G_2)=1$. We fix an identification
$\wh{\G}_1= \wh{\G}_2$. Then $N_1$ and $N_2$ are dense in the kernel of the canonical epimorphism $\wh\G_1
\rightarrow\wh\Z$ described in (\ref{equ}). And by Lemma \ref{l:full}, this kernel is isomorphic to $\wh{N}_i$ for $i=1,2$.
\end{proof}

\subsection{The groups $\G_\phi=F_r\rtimes_\phi\Z$}

The action of ${\rm{Aut}}(F_2)$ on $H_1(F_2,\Z)$ gives an epimorphim $\aut(F_2)\to {\rm{GL}}(2,\Z)$ whose kernel 
is the group of inner automorphisms. The isomorphism type of $\G_\phi$ depends only on the 
conjugacy class of the image of $\phi$
in $\out(F_2)={\rm{GL}}(2,\Z)$, and we often regard $\phi$ as an element of ${\rm{GL}}(2,\Z)$.
We remind the reader that finite-order
elements of ${\rm{GL}}(2,\Z)$ are termed {\em{elliptic}},  infinite-order elements with an eigenvalue of absolute value bigger than $1$ are {\em{hyperbolic}}, and the other infinite-order elements are {\em{parabolic}}. 
Note that an element of infinite order in $\SL(2,\Z)$
is hyperbolic if and only if $|\tr(\phi)|>2$.

Every automorphism $\phi$ of $F_2$ sends $[a,b]$ to a conjugate of
$[a,b]^{\pm 1}$
and hence can be realised as an automorphism of the once-punctured torus (orientation-preserving or reversing according to whether $\det\phi = \pm 1$). 
 
\subsection{Features distinguished by $\wh{\G}_\phi$} 

\begin{proposition}\label{p:unique}
Suppose  $\wh{\G}_{\phi_1}\cong\wh\G_{\phi_2}$.
If $\phi_1$ is hyperbolic then $\phi_2$ is hyperbolic and has the same eigenvalues as $\phi_1$
(equivalently, $\det \phi_1=\det \phi_2$ and $\tr\phi_1=\tr\phi_2$).
\end{proposition}
 
We break the proof of this proposition into three lemmas.
First, for
the assertion about $\det\phi$, we use the short exact sequence (\ref{equ}).

\begin{lemma}\label{l:det}
If $b_1(\G_{\phi_1})=1$ and $\wh{\G}_{\phi_1}\cong\wh{\Gamma}_{\phi_2}$, then  $\det\phi_1 = \det\phi_2$.
\end{lemma}

\begin{proof} We use the shorthand $\G_i={\G_{\phi_i}}$, fix an identification $\wh\G_1=\wh\Gamma_2$,
and write $N$ for the kernel of the canonical map $\wh\G_i\to\wh\Z$. The canonical map  
$\G_1\to\G_1/[\G_1,\G_1]\G_1^3$ induces an epimorphism $N\to (\Z/3)^2$, whose restriction to $\G_2$
has kernel $[\G_2,\G_2]\G_2^3$. The action of $\wh\G_i$ by conjugation on $N$ induces a map $\wh\Z\to {\rm{GL}}(2,\Z/3)$,
with cyclic image generated by the reduction of $\phi_i$ for $i=1,2$. Thus $\det\phi_1=\det\phi_2$ is determined by whether the image
of $\wh\Z$ lies in ${\rm{SL}}(2,\Z/3)$ or not.
\end{proof}

\begin{lemma}\label{l:L1}
If $b_1(\G_{\phi_1})=1$ and $\wh{\G}_{\phi_1}\cong\wh{\Gamma}_{\phi_2}$,
 then $\wh{\G}_{\phi_1^r}\cong\wh{\Gamma}_{\phi_2^r}$ for all $r\ne 0$.
\end{lemma}

\begin{proof} $\G_{\phi^r}\cong \G_{\phi^{-r}}$, so assume $r>0$. We again
consider the canonical map from $\wh{\G}_1=\wh{\Gamma}_2$
to $\wh\Z$. Let $N_r$ be the kernel of the composition of this map and 
$\wh\Z\rightarrow \Z/r$. Then $N_r$ is the closure of $F_2\rtimes_{\phi_i}r\Z <\G_{\phi_i}$
for $i=1,2$. Thus $\wh\G_{\phi_1^r}\cong N_r \cong \wh\Gamma_{\phi_2^r}$.
(Here we have used the fact that the profinite topology on any finitely generated group induces the full profinite topology on any subgroup of finite index.)
\end{proof}

\begin{lemma}\label{l:torsion}${}$
\begin{enumerate}
\item  $b_1(\Gp)=1$ if and only if $1 + \det \phi \neq  \tr \phi$.
\item If $b_1(\Gp)=1$ then $H_1(\Gp,\Z)\cong\Z\oplus T$, where $|T|= |1+\det\phi - \tr\phi|$.
\end{enumerate}
\end{lemma}

\begin{proof} By choosing a representative $\phi_*\in {\rm{Aut}}(F_2)$, we get a presentation for $\Gp$,
$$
\< a,b, t \mid tat^{-1} = \phi_*(a),\, tbt^{-1}=\phi_*(b)\>.
$$
By abelianising, we see that $H_1(\Gp,\Z)$ is the direct sum of $\Z$ (generated by the image of $t$)
and $\Z^2$ modulo the image of $\phi-I$. The image of $\phi-I$ has  finite index if and only if $\det(\phi-I)$
is non-zero, and a trivial calculation shows that this determinant is $1 - \tr\phi + \det\phi$. If the index is
finite, then the quotient has order $|\det(\phi - I)|$.
\end{proof}
 
\begin{corollary}\label{c:c1} $b_1(\Gamma_{\phi^r})=1$ for all $r\neq 0$ if and only if $\phi$ is hyperbolic.
\end{corollary}

\begin{proof} 
If $\phi$ is hyperbolic, then $\phi^r$ is hyperbolic for
  all $r\neq 0$; in particular  $\tr\phi^r \neq 0$, and $|\tr\phi^r |>2$ if $\det \phi^r=1$.
Thus the result follows from
  Lemma \ref{l:torsion}(1),  in the  hyperbolic case.  If $\phi$ is elliptic,
  then $\phi^r=I$ for some $r$, whence $b_1(\Gamma_{\phi^r})=3$.  If
  $\phi$ is parabolic, then there exists $r>0$ 
so that $\phi^r$ had determinant $1$ and is conjugate
in $\GL(2,\Z)$ to 
$\begin{pmatrix} 1& n\cr 0& 1\cr\end{pmatrix}$ with $n>0$. In particular,
$b_1(\Gamma_{\phi^r})=2$. 
\end{proof}
  
\noindent{\bf{ Proof of Proposition \ref{p:unique}.}} Lemma \ref{l:L1} and Corollary \ref{c:c1} imply that
$\G_{\phi_2}$ is hyperbolic, Lemma \ref{l:det} shows that $\det \phi_1=\det \phi_2$, and then
Lemma \ref{l:torsion}(2) implies that $\tr\phi_1=\tr\phi_2$ (since $H_1(\Gamma_{\phi_1},\Z)\cong H_1(\Gamma_{\phi_2},\Z)$).
\qed

\begin{remark}\label{r:torus}
So far, the only quotients that we have used to explore the profinite completion of
$\Gp$ are the abelian quotients of $F_2\rtimes_{\phi^r}\Z$. Since these all factor through $A_{\phi^r}=\Z^2\rtimes_{\phi^r}\Z$, we were
actually extracting information about $\phi$ from the groups $\wh{A}_{\phi^r}$. Up to isomorphism, 
such a completion $\wh{A}_{\psi}$ 
is determined by the {\em local} conjugacy class of $\psi$, meaning $\wh{A}_{\psi}\cong\wh{A}_{\psi'}$  if and 
only if the image of $\psi$ is conjugate to the image of  $\psi'$  in ${\rm{GL}}(2,\Z/m)$ for all integers $m>1$. 
Local conjugacy does not imply that $\psi$ and $\psi'$ are conjugate in ${\rm{GL}}(2,\Z)$: Stebe \cite{stebe}
proved that 
$$
\begin{pmatrix} 188& 275\cr 121& 177\cr\end{pmatrix}\ \ \hbox{ and } \ \ \ \ \begin{pmatrix} 188& 11\cr 3025& 177\cr\end{pmatrix}
$$
have this property 
and Funar \cite{Fu} described infinitely many such pairs. The corresponding torus bundles over the circle provide pairs of ${\rm{Sol}}$ manifolds 
whose fundamental groups are not isomorphic but have the same profinite completion \cite{Fu}.
In a forthcoming article with Henry Wilton we shall prove that, in contrast,  {\em{punctured}} torus bundles over the circle are determined
up to homeomorphism by the profinite completions of their fundamental groups.  
\end{remark}

\begin{remark} \label{r:cft}
The conclusion of Proposition \ref{p:unique} could also be phrased as saying that $\phi_1$ and $\phi_2$ lie in the same conjugacy class
in ${\rm{GL}}(2,\Q)$. When intersected with ${\rm{GL}}(2,\Z)$, this will break into a finite number of conjugacy
classes; how many can be determined using  class field theory \cite{tt}. 
\end{remark}

\subsection{The figure-eight knot}

 For small examples, one can calculate the conjugacy classes in ${\rm{GL}}(2,\Z)$ with a given trace and determinant by hand. For example, up to conjugacy
in ${\rm{GL}}(2,\Z)$ the only matrix with trace $3$ and determinant $1$ is $\begin{pmatrix} 2& 1\cr 1& 1\cr\end{pmatrix}$.
This calculation yields the following consequence of  Proposition \ref{p:unique} and Lemma \ref{l:sameHat}.

We retain the notation established in the introduction: $\K$ is the
figure-eight knot and $\Ge=\pi_1(\Sh^3\ssm\K)$.

\begin{proposition}\label{p:8unique} Let $\G=F\rtimes_\phi\Z$, where $F$ finitely generated and free.
 If $\wh{\G}\cong\wh\Ge$, then $F$ has rank two, $\G\cong\Ge$,  and $\phi$ is conjugate to 
 $\begin{pmatrix} 2& 1\cr 1& 1\cr\end{pmatrix}$ in any identification 
of ${\rm{Out}}(F)$ with ${\rm{GL}}(2,\Z)$.
\end{proposition}

\subsection{Uniqueness for the trefoil knot and the Gieseking manifold}

From Lemma \ref{l:torsion} we know that  
the only monodromies $\phi\in {\rm{GL}}(2,\Z)$ for which $H_1(\G_\phi,\Z)\cong\Z$ are those
 for which $(\tr(\phi), \det \phi)$
 is one of $(1,1),\ (1,-1),\, (3,1)$. We have already discussed how 
 the last possibility determines the figure-eight knot. Each of the other possibilities
 also determines a unique conjugacy class in ${\rm{GL}}(2,\Z)$,
 represented by $$
\begin{pmatrix} 1 & -1\cr 1& 0\cr \end{pmatrix}\ \ \ \ \hbox{and}\ \ \ \ 
\begin{pmatrix} 1 & 1\cr 1& 0\cr \end{pmatrix},
$$
respectively. The punctured-torus bundle with the first monodromy is the complement of the trefoil knot, and
the second monodromy produces the Gieseking manifold, which is the unique non-orientable 3-manifold whose
oriented double cover is the complement of the figure-eight knot.

Note that the first matrix is elliptic while the second is hyperbolic. 
The following proposition records two consequences of this discussion. For item (2) we
need to appeal to Lemma \ref{l:sameHat}.

\begin{proposition}\label{p:trefoil-v-FZ}
\begin{enumerate}
\item
The only groups of the form $\Lambda=F_2\rtimes\Z$ with $H_1(\Lambda,\Z)=\Z$ are the fundamental
groups of (i) the figure-eight knot complement, (ii) the trefoil knot, and (iii) the  Gieseking manifold.
\item
Let $\Lambda$ be one of these three groups and let $F$ be a free group.
If $\G=F\rtimes\Z$ and $\wh\G\cong\wh\Lambda$, then $\G\cong\Lambda$.
\end{enumerate}
\end{proposition}

\begin{example} Instead of appealing to the general considerations in this section, one can distinguish the
profinite completions of the groups for the figure-eight knot and the trefoil knot directly. Indeed, 
setting $y^2=1$ in the standard presentation 
$$
\Pi = \< x, y\mid yxy^{-1}xy=xyx^{-1}yx   \>
$$
one sees that $\Pi$ maps onto the dihedral group $D_{10}$. But $T$, the fundamental group of the trefoil
knot, cannot map onto $D_{10}$, because $D_{10}$ is centreless and has no elements of order $3$, whereas
$T=\<a,b\mid a^2=b^3\>$ is a central extension of $\Z/2\ast \Z/3$.
\end{example}

\subsection{Finite Ambiguity}

We do not know if all free-by-cyclic groups can be  distinguished from one another by their profinite completions
(cf.~Remark \ref{r:torus}), but the analysis in the previous subsection enables us to show that the ambiguity 
in the case when the free group has rank $2$ is at worst finite.

\begin{proposition}\label{p:finite} For every $\phi\in{\rm{GL}}(2,\Z)$, there exist only finitely many
conjugacy classes $[\psi]$ in ${\rm{GL}}(2,\Z)$ such that $\wh{\G}_\phi\cong \wh\G_\psi$.
Moreover, all such $\psi$ are of the same type, hyperbolic, parabolic, or elliptic.
\end{proposition}

\begin{proof} If $\phi$ is hyperbolic, this follows from 
Proposition \ref{p:unique}. If $\phi$ is
 parabolic then either it has trace $2$ and $b_1(\G_\phi)=2$, or trace $-2$ in which case $b_1(\G_\phi)=1$.
In the former case, $\phi$ is conjugate in ${\rm{GL}}(2,\Z)$ to 
$\begin{pmatrix} 1& n\cr 0& 1\cr\end{pmatrix}$ with $n>0$, and $H_1(\G_\phi,\Z)\cong\Z^2\oplus \Z/n$,
so $n$ is determined by $H_1(\G_\phi,\Z)$, hence by $\wh{\G}_\phi$.
In the case where the trace is $-2$, we have  $b_1(\G_\phi)=1$ and
Lemma \ref{l:L1} reduces us to the previous case.

For the elliptic case, since the possible orders are $2$, $3$, $4$ or $6$ and
there are only finitely many conjugacy classes in ${\rm{GL}}(2,\Z)$
of elements of such orders, we are reduced to distinguishing 
elliptics from non-elliptics by means of $\wh\G_\phi$.
If $\phi$ is elliptic and $b_1(\G_\phi)=1$, then Lemma \ref{l:L1} completes the proof, because for 
non-elliptics $b_1(\G_{\phi^r})$ is never greater than $2$, whereas for elliptics it becomes $3$. 
The only other possibility, up to conjugacy, is $\begin{pmatrix} 0& 1\cr 1& 0\cr\end{pmatrix}$,
which cannot be distinguished from parabolics such as
$\begin{pmatrix} 1& 1\cr 0& 1\cr\end{pmatrix}$ by means of 
$H_1(\G_\phi,\Z)$. However, these can be distinguished on passage to subgroups of
finite index, since for $\phi=\begin{pmatrix} 0& 1\cr 1& 0\cr\end{pmatrix}$,
$b_1(\G_{\phi^2})=3$, while for parabolics  ,
$b_1(\G_{\phi^2})=2$.  
\end{proof}
 
\subsection{Passing to finite-sheeted covers}

It is easy to see that a subgroup of finite index in a free-by-cyclic
group is itself free-by-cyclic.  Finite extensions of $F_r\rtimes\Z$,
even if they are torsion-free, will not be free-by-cyclic in general,
but in the setting of the following lemma one can prove something in
this direction.  This lemma is useful when one wants to prove that a
manifold is a punctured-torus bundle by studying finite-sheeted
coverings of the manifold.

\begin{lemma}\label{l:fi} Let $\Lambda$ be a torsion-free group and let $\G<\Lambda$ be a subgroup of index $d$.
Suppose that $b_1(\Lambda)=b_1(\G)=1$. 
If $\G\cong F_2\rtimes_\phi\Z$, then $\Lambda\cong F_2\rtimes_\psi\Z$, with $\phi=\psi^d$.
\end{lemma}

\begin{proof} As $b_1(\G)=b_1(\Lambda)=1$, there are unique normal subgroups 
$I\subset \G$ and $J\subset \Lambda$ such that $\G/I\cong \Lambda/J\cong\Z$.
The image of $\G$ has finite index in $\Lambda/J$, so $I=J\cap\G$. But $I\cong F_2$, and $F_2$ is not a proper subgroup of finite index in any torsion-free
group (because such a group would be free with euler characteristic dividing $-\chi(F_2)=1$). Thus $I=J=F_2$, and the image of $\G/I$
in $\Lambda/J\cong \Z$ has index $d$.
\end{proof}

\section{Profinite completions of 3-manifold groups} 
We shall make use of several results about the profinite
completions of 3-manifold groups. We summarize these in the following theorem, and quote their origins in the proof.  
Note that in what follows
by the statement {\em a compact 3-manifold $M$ contains an incompressible Klein
bottle ${\bf K}$}, we shall 
mean that the induced map $\pi_1({\bf K})\hookrightarrow \pi_1(M)$ is injective.

\begin{theorem}
\label{summary}
Let $X$ be a compact connected 3-manifold.  Then:
\begin{enumerate}
\item If  $\widehat{\pi_1(X)}\cong \widehat{\pi_1(M)}$, then 
$H_1(X,\Z)\cong H_1(M,\Z)$.
\item $\pi_1(X)$ is good.
\item $b_1^{(2)}(X)=0$ if and only if $\pi_1(X)$ is virtually infinite cyclic 
or $X$ is aspherical (hence irreducible) and 
$\partial X$ consists of a (possibly empty) disjoint union of 
tori and Klein bottles.
\item If $X$ is closed
and $\G$ is either free-by-cyclic or else the fundamental group of a 
non-compact finite volume hyperbolic 3-manifold, then
$\widehat{\pi_1(X)}$ and $\widehat{\G}$ are not isomorphic.
\item Suppose that $X$ is a Seifert fibred space and that $\G$ is either  the fundamental group of a 
non-compact finite volume hyperbolic 3-manifold or else of the form $F_r\rtimes_\phi\Z$ where $r\geq 2$ and
 $[\phi]$ has infinite order in ${\rm{Out}}(F_r)$. Then
$\widehat{\pi_1(X)}$ and $\widehat{\G}$ are not isomorphic. 
\end{enumerate}
\end{theorem}

\begin{proof} (1) is a special case of Lemma \ref{l:abel}. 
Using \cite{Ag} and \cite{Wi}, goodness of compact 3-manifold groups
follows from \cite{Cav} (see also \cite{Re}
Theorem 7.3 and \cite{AFW} \S  6G.24).  

Part (3) follows from the calculations of 
Lott and L\"uck \cite{LL}[Theorem 0.1].
Their theorem is stated only for orientable manifolds but this is not a serious problem because,
by L\"uck approximation (which we used to  define $b_1^{(2)}(X)$), 
if $X$ is a non-orientable compact 3-manifold with infinite fundamental group
and $Y\rightarrow X$ is its orientable double cover, then  $b_1^{(2)}(Y)=2\, b_1^{(2)}(X)$.

To prove (4) we follow the proof of \cite{Re} Theorem 8.3.
Item (3) and Proposition \ref{p:gab} tell us that
$b_1^{(2)}(\G)=0$, so by Proposition \ref{dense} we have
$b_1^{(2)}(X)=0$. 
Moreover,
every subgroup of finite index in $\G$ has non-cyclic finite quotients,
so the opposite implication in (3) tells us that
$X$ is aspherical.  We know that 
$\pi_1(X)$ is good (item (2)), so
since $X$ is closed, $\cd(\wh{\pi_1N})=3$ by Lemma \ref{l:cd}. But
$\G$ is also good (by (2) or Lemma \ref{l:extend}), and
$\cd(\wh{\G})=2$, by Lemma \ref{l:cd}.

To prove the last part, we  argue  as follows. 
First, since
every subgroup of finite index in $\G$ has infinitely many finite quotients
that are not solvable,  we can
quickly eliminate all possibilities for $X$ apart from those with
geometric structure modelled on ${\H}^2\times {\R}$ and
$\widetilde{\SL}_2$. In these two remaining cases, the projection of $X$ onto its base orbifold  gives
a short exact sequence of profinite groups
$$
1\to \overline{Z} \to \wh{\pi_1X}\to \wh{\Lambda}\to 1,
$$
where $Z$ is the infinite cyclic centre of $\pi_1X$ and $\Lambda$  is a non-elementary discrete
group of isometries of ${\H}^2$.
(In fact, following the discussion in \S 2, we know that $\overline{Z}=\wh{Z}$, because $\pi_1(X)$ is LERF, but we do not need this.)
Note that $\overline{Z}$ is central in $\wh{\pi_1(X)}$, because if $[z,g]\neq 1$ for some $z\in\overline{Z}$ and $g\in\wh{\pi_1X}$,
then for some finite quotient $q:\wh{\pi_1(X)}\to Q$ we would have 
$[q(z),q(g)]\neq 1$, which
 would contradict the centrality of
$Z$, since $q(\overline{Z}) = q(Z)$ and $q(\wh{\pi_1(X)}) = q({\pi_1(X)})$. 

It is easy to see that the fundamental group of a finite volume hyperbolic 3-manifold has trivial centre.
A trivial calculation shows that $F_r \rtimes_\phi \Z$ also has trivial centre, except  when a power of $\phi$ is
an inner automorphism. (In the exceptional case,  if $t$ is the generator of the $\Z$ factor and $\phi^r$ is conjugation by $u\in F_r$,
then $t^{-r}u$ is central.)

Since $\G$ has trivial centre, its image under any isomorphism $\wh{\G}\to \widehat{\pi_1(X)}$ would intersect
$\overline{Z}$ trivially, and hence project to a   dense subgroup of $\widehat{\Lambda}$.
But using L\"uck approximation, it is easy to see that $b_1^{(2)}(\Lambda)\neq 0$
since $\Lambda$ is either virtually free of rank at least $2$
or virtually a surface group of genus
at least $2$. And in the light of
Theorem \ref{summary}(3) and Proposition \ref{p:gab}, this would contradict Proposition \ref{dense}.
\end{proof}

\begin{remark}
In the setting of Theorem \ref{summary}(5), when $\G$ is the fundamental group
of a non-compact finite volume hyperbolic 3-manifold, one  can use
Proposition 6.6 of \cite{WZ} together with
\cite{Wi} to prove the stronger statement that $\widehat{\G}$ has trivial
centre.\end{remark}

\begin{corollary}\label{c:notClosed}
Let $X$ be a compact orientable 3-manifold. If $\wh{\pi_1(X)}\cong \wh{\G}$, where $\G=F_r\rtimes\Z$,
then $X$ is irreducible and its boundary is a union of $t$ incompressible tori where
$1\le t\le b_1(\G)$.
\end{corollary}

\begin{proof} 
Theorem \ref{summary} (3), (4), (5) 
tells us that $X$ is irreducible and has non-empty toral boundary. If one of the boundary
tori were compressible, then by irreducibilty $X$ would be a solid torus, which it is not since $\G$
has non-abelian finite quotients and $\wh{\G}=\wh{\pi_1(X)}$. 

The upper bound on the number of tori comes from the well-known ``half-lives, half-dies'' phenomenon described in the following 
standard consequence of Poincare-Lefschetz duality. \end{proof}

\begin{proposition}
\label{halfliveshalfdies}
Let $M$ be a compact orientable 3-manifold with non-empty boundary. 
The rank of the image of 
$$H_1(\partial M, \Z)\rightarrow H_1(M, \Z)$$
is $b_1(\partial M)/2$.
\end{proposition}

It is more awkward to state the analogue of
Corollary \ref{c:notClosed} for non-orientable manifolds $M$.  
One way
around this is to note that an index-2 subgroup of
$F_r\rtimes_{\phi}\Z$ is either $F_r\rtimes_{\phi^2}\Z$ or else is of
the form $F_{2r-1}\rtimes\Z$, so the orientable double cover $X$ is of
the form described in Corollary \ref{c:notClosed} and each boundary
torus in $X$ either covers a Klein bottle in $\partial X$ or a torus
(1-to-1 or 2-to-1).

\section{The Proof of Theorems B and C} 
 
The results in this section form the technical heart of the paper. In the statement of Theorem \ref{t:fibres},
we assumed that $b_1(M)=1$, so the following theorem applies in that setting. Indeed, in the light of  
Corollary \ref{c:notClosed} and the comment that follows Proposition \ref{halfliveshalfdies},
Theorem \ref{t:realB} completes the proof of Theorem \ref{t:fibres}. 

\begin{theorem}\label{t:realB}
Let $M$ be a compact, irreducible, 3-manifold whose boundary is a non-empty union of incompressible tori and Klein bottles. Suppose that there is an
isomorphism  $\wh{\pi_1(M)}\to\wh{F\rtimes\Z}$ such that $\pi_1M$ has cyclic image under the composition
$$
\pi_1M\to\wh{\pi_1M}\to\wh{F\rtimes\Z}\to\wh{\Z},
$$
where $F$ is finitely generated and free, and $\wh{F\rtimes\Z}\to\wh{\Z}$ is induced by the obvious surjection $F\rtimes\Z\to\Z$. Then,
$M$ fibres over the circle with compact fibre (equivalently, 
$\pi_1(M)$ is of the form $F_r\rtimes\Z$).
\end{theorem}

\begin{remark}\label{b_1need}
Given an arbitrary isomorphism  $\wh{\pi_1(M)}\to\wh{F\rtimes\Z}$, the image of $\pi_1M$ in $\wh{\Z}$ will be a finitely generated, dense,
free abelian group, but one has to contend with the fact that it might not be cyclic. In Theorem \ref{t:fibres} we overcame this by requiring 
$b_1(M)=1$. Boileau and Friedl
\cite{BF} avoid the same problem by restricting attention to  isomorphisms $\wh{\G}_1\to\wh{\G}_2$ that  induce an isomorphism on $H^1(\G_i,\Z)$.
In Theorem \ref{t:realB} we remove the difficulty directly with the hypothesis on $\pi_1(M)\to\wh{\Z}$.
\end{remark}

Theorem \ref{t:realB} is proved by applying the following result to the map from $\pi_1M$ to 
$\wh{\Z}$. In more detail,  
in Theorem \ref{t:realB} we assume that we have an isomorphism
$\wh{\pi_1(M)}\to\wh{F\rtimes\Z}$, which we compose with
$\wh{F\rtimes\Z}\to\wh{\Z}$ to obtain an epimorphism of profinite groups $\Phi:\wh{\pi_1M}\to\wh{\Z}$.
Let $f$ denote the restriction of $\Phi$ to $\pi_1M$.
The kernel of $\Phi$ is isomorphic to the free profinite group $\wh{F}$; in particular
it has cohomological dimension 1. From Proposition \ref{cdbound}  it follows 
that the closure of  $\ker~f$ also 
has cohomological dimension 1. And we have assumed that the image of $f$
is isomorphic to $\Z$, so Theorem \ref{t:freedmen} implies that $\ker~f$ is finitely generated and free.

\begin{theorem}\label{t:freedmen}
Let $M$ be a compact, irreducible, 3-manifold with non-empty boundary 
which is a non-empty union of incompressible tori and Klein bottles. 
Let $f:\pi_1(M)\to\Z$ be an epimorphism. 
Then, either $\ker~f$ is finitely generated and free (and $M$ is fibred) or else
the closure of $\ker~f$ in $\wh{\pi_1M}$ has cohomological dimension at least $2$.
\end{theorem} 

\begin{proof} In Theorem 3 of \cite{FF}, Freedman and Freedman prove
that if $M$ is a compact 3-manifold with each boundary component
either a torus or a Klein bottle and $f:\pi_1(M)\to\Z$ has a
non-trivial restriction to each boundary component, then either (1)
the infinite cyclic covering $M^f\to M$ corresponding to $f$ is
homeomorphic to a compact surface times $\R$, or (2) the cover $M^f$
contains a closed 2-sided embedded incompressible surface $S$.

In Theorem \ref{t:freedmen}, we are assuming that the boundary
components are incompressible, so if the restriction of $f$ to some
component $T$ were trivial, then $\pi_1(T)$ would lie in the kernel of
$f$. Let $A\cong \Z^2$ be a subgroup of finite index in
$\pi_1(T)$. From \cite{Ham}, we know that $\pi_1(M)$ induces the full
profinite topology on $A$. Therefore, by Lemma \ref{l:cd}, the closure
of $\ker~f$ in $\wh{\pi_1(M)}$ has cohomological dimension at least
$2$, as required.  Thus we may assume that $f$ satisfies the
hypotheses of Theorem 3 of \cite{FF}. If alternative (1) of that
theorem holds, then the kernel of $f$ is a finitely generated free
group and we are done. So we assume that alternative (2) holds and
consider the closed embedded 2-sided incompressible surface $S$.

Since $S$ is compact, it does not intersect its translates by suitably high powers of the basic deck 
transformation of $M^f$,
so we can factor out by such a power to obtain a finite-cyclic covering $N\to M$ in which $S$ embedded.
The proof will be complete if we can prove that $\pi_1(N)$ (equivalently 
$\pi_1(M)$) induces the full
profinite topology on $\pi_1(S)$ (since, by construction, $\pi_1(S)$ 
is contained in $\ker~f$). 

It will be convenient in what follows to pass to a finite cover of $N$
and $S$ (which we continue to call $N$ and $S$) so that both are orientable
and $S$ continues to be embedded. Standard separability properties ensure
that if $\pi_1(M)$ induces the full profinite topology on $\pi_1(S)$, then
this will get promoted to the original surface. Henceforth we assume
$N$ and $S$ are orientable.

In more topological language, what we must prove is the following. 

\begin{proposition}\label{l:key}
If $p:S_1\to S$ is the finite-sheeted covering corresponding to an arbitrary finite-index subgroup of $\pi_1(S)$,
then there is a finite-sheeted covering $q:N_1\to N$ so that $p$ is the restriction of $q$ to a connected
component of $q^{-1}(S)$.
\end{proposition}

This will be proved in the next section, thus completing the proof of Theorem
\ref{t:freedmen}.\end{proof}

\subsection{Surface separability, controlled Dehn filling and aspiral surfaces}

What we need in Proposition \ref{l:key} is close to a recent theorem of
Przytycki and Wise in \cite{PW2} (see also \cite{PW1} and \cite{PW}).
They prove that for every closed incompressible surface $S$ embedded
in a compact 3-manifold $M$, the image of $\pi_1(S)$ is closed in the 
profinite topology of
$\pi_1(M)$. But we need more (recall \S 2.1):
all finite-index subgroups of  $\pi_1(S)$ have to be closed in $\pi_1M$ as
well. In order to prove their results, Przytycki and Wise developed a
technology for merging finite covers of the blocks in the JSJ
decomposition of the manifold \cite{PW}, \cite{PW2}, \cite{PW1}.  Liu
\cite{Liu} and his coauthors \cite{DLW} refined this technology to
establish further results, and Liu's refinement in \cite{Liu} will
serve us well here. But since his criterion applies only to closed
manifolds, we have to perform Dehn filling on $N$ to get us
into this situation.

We introduce some notation. Assume that $N$ has $t$ boundary
components all of which are incompressible tori.  For a choice of slopes
$r_1, \ldots ,r_t$, denote by $N(r_1,\ldots r_t)$ the manifold obtained by
performing $r_i$-Dehn filling on the torus $T_i$ for $i=1,\dots,t$.

\begin{lemma}
\label{surger}
Let $N$ and $p:S_1\to S$ be as above. Then:
\begin{enumerate}
\item  There exist infinitely many collections of slopes $(r_1,\ldots r_t)$
so that the Dehn fillings $N(r_1,\ldots ,r_t)$ of $N$
are irreducible and the image of 
$S$ in $N(r_1,\ldots ,r_t)$ remains incompressible (and embedded). 
\item There exists a finite cover $q:N'\rightarrow N(r_1,\ldots ,r_t)$ and an embedding $\iota:S_1\to N'$
such that $q\circ \iota =p$.
\end{enumerate}
\end{lemma}

\noindent{\bf{Proof of Proposition \ref{l:key}.}}
Since $S\subset N(r_1,\ldots ,r_t)$ lies in the
complement of the filling core curves of the filled tori, $S_1\subset N'$ lies in the complement
of the preimage of these curves, so by deleting them
we get the desired finite cover $N_1\rightarrow N$.
\qed\\[\baselineskip]
\noindent{\bf Proof of Lemma \ref{surger}:}~ 
For part (1), incompressibility can be deduced from
\cite{CGLS} Theorems 2.4.2 and 2.4.3, as we shall now explain.  

One possibility in \cite{CGLS} is that $N$ is homeomorphic to
$T^2\times I$. 
But we can exclude this possibility because our $N$ contains a closed incompressible surface 
that is not boundary parallel whereas $T^2\times I$ contains no such surface

We now proceed to apply \cite{CGLS} one torus at a time. 
 To ensure that $S$ remains incompressible upon $r_1$-Dehn filling on
$T_1$, we simply arrange by \cite{CGLS} to choose (from infinitely
many possibilities) a slope $r_1$ that has large distance (i.e. $>2$)
from any slope for which filling compresses $S$, as well as from any
slope that co-bounds an annulus with some essential simple closed curve
on $S$.  Repeating this for each torus $T_i$ in turn ensures that $S$ remains
incompressible in $N(r_1,\ldots ,r_t)$.

To ensure that $N(r_1,\ldots ,r_t)$ is irreducible, we make the
additional stipulation (possible by \cite{GL}) that for $i=1,\ldots ,t-1$ each
$r_i$-Dehn filling avoids the three possible slopes that result in a
reducible manifold upon Dehn filling the torus $T_i$. (Note that
\cite{GL} is stated only for manifolds with a single torus boundary
component but the proof and result still apply to our context.) 
Now, for the final torus $T_t$, we simply choose $r_t$-Dehn
filling so as to avoid the finite number of boundary slopes, and hence
we avoid any essential embedded planar surface that could give rise to
a reducing sphere in $N(r_1,\ldots ,r_t)$.

For the proof of (2), we need to recall some of Liu's terminology
\cite{Liu}. Theorem 1.1 of \cite{Liu} includes the statement:

\smallskip

{\em A surface $\Sigma$ is aspiral in the almost fibre part if and only if $S$ is virtually
essentially embedded.}

\smallskip

\noindent Since $S$ is embedded and incompressible in $N(r)$, it is
aspiral. Assertion (2) of the lemma is equivalent to the statement
that $S_1$ is virtually essentially embedded, so what we must argue is
that $S_1$ is aspiral.

Let $M$ be a closed orientable 3-manifold and $\Sigma$ a closed essential
surface (of genus $\geq 1$) {\em{immersed}} into $M$. The JSJ decomposition of $M$
induces a canonical decomposition of $\Sigma$. Following
\cite{Liu}, the {\em almost fibre part} of $\Sigma$, denoted $\Phi(\Sigma)$ is
the union of all its horizonal subsurfaces in the Seifert manifold pieces, together 
with the geometrically infinite (i.e.~virtual fibre) subsurfaces
in the hyperbolic pieces. Note that, by definition,
for any finite covering $\Sigma_1\rightarrow \Sigma$,
the preimage of the almost fibre part $\Phi(\Sigma)$ is $\Phi(\Sigma_1)$.

To define aspiral we recall Liu's construction
of the {\em spirality character}  \cite{Liu}.
Given a principal ${\Q}^*$-bundle $P$ over $\Phi(\Sigma)$, the
spirality character of $P$, denoted $s(P)$, is the element of 
$H^1(\Phi(S),{\Q}^*)$ constructed as follows.  For any closed
loop $\alpha:[0,1]\rightarrow \Phi(\Sigma)$ based at $x_0\in \Phi(\Sigma)$,
 each choice of lift for $x_0$ determines a lift $\tilde{\alpha}:[0,1]\rightarrow P$,
and a ratio $\tilde{\alpha}(1)/\tilde{\alpha}(0)\in {\Q}^*$.
This ratio depends only on $[\alpha ]\in H_1(\Phi(\Sigma),{\Z})$ and so induces
a homomorphism $s(P): H_1(\Phi(\Sigma),{\Z})\rightarrow {\Q}^*$; i.e.
an element of $H^1(\Phi(\Sigma),{\Q}^*)$.  The principal bundle $P$ is called 
{\em aspiral} if $s(P)$ takes only the values $\pm 1$ for all 
$[\alpha ]\in H_1(\Phi(\Sigma),{\Z})$. The proof of Theorem 1.1
of \cite{Liu} involves the construction of a particular
principal ${\Q}^*$-bundle $\mathcal{H}$
(see Proposition 4.1 of \cite{Liu}). 
Then, $\Sigma$ is defined to be {\em{aspiral in the almost fibre part}} if 
the bundle $\mathcal{H}$ is aspiral.

For brevity, set $s=s(\mathcal{H})$. To see that aspirality passes to finite
covers of $\Sigma$, one argues as follows. 
First, as above, we note that for a finite covering
$\Sigma_1\rightarrow\Sigma$, the almost fibre part $\Phi(\Sigma_1)$ is
the inverse image of $\Phi(\Sigma)$.  Next we observe that the
naturality established in the proof of \cite[Proposition 4.1]{Liu}
(more specifically Formula 4.5 and the paragraph following it), shows
that $s$ pulls back to the spirality character $s_1$ of
$\Phi(\Sigma_1)$. Indeed the bundle $\mathcal{H}$ in the definition of
$s$, which is independent of the data chosen in the construction,
builds-in the data for all finite covers of $\Sigma$; see \S 4.2 of
\cite{Liu}.

Thus, in our setting,  $S_1$ is aspiral in $\Phi(S_1)$ 
as required.\qed\\[\baselineskip]
For emphasis, we repeat the key fact that we have extracted from
\cite{PW2} and \cite{Liu}.

\begin{theorem} If $M$ is a closed orientable 3-manifold and $S\subset M$ is a closed embedded incompressible
surface, then $\pi_1(M)$ induces the full profinite topology on $\pi_1(S)$.
\end{theorem}
 
\section{Profinite rigidity for the figure-knot complement}

As before, let $\K$ be the figure-eight knot and let $\Ge = \pi_1(\Sh^3\ssm\K)$.

\begin{theorem}\label{t:main}
Let $M$ be a compact connected 3-manifold. If $\wh{\pi_1(M)}=\wh{\Ge}$, 
then $M$ is homeomorphic to $\Sh^3\ssm\K$.
\end{theorem}

\begin{proof}  
Theorem \ref{t:fibres} tells us that $M$ is a fibred manifold with fundamental group of the form $F_r\rtimes\Z$,
and Proposition \ref{p:8unique} then tells us that $r=2$ and  $\pi_1(M)\cong\Ge=F_2\rtimes_\phi\Z$. 
The fundamental group of the
fibre in $M$ is the kernel of the unique map $\pi_1(M)\to\Z$, 
so it is free of rank $2$. The
only compact surface with fundamental group $F_2$ that supports a hyperbolic automorphism is the punctured
torus, so $M$ is the once-punctured torus bundle with holonomy $\phi$, i.e.~the complement
of the figure-eight knot.
\end{proof}

Similarly, using Proposition \ref{p:trefoil-v-FZ} one can show that 
the complement of the trefoil knot and the Gieseking manifold are determined up to homeomorphism by their
fundamental groups.

\section{Related Results}

We have built our narrative around the complement of the figure-eight knot, but our arguments
also establish results for larger classes of manifolds. In this section we record some of these
results. 
We write $M_\phi$ to denote the punctured torus bundle with monodromy $\phi\in{\rm{GL}}(2,\Z)$.  

\begin{theorem}
\label{finitelymany}
Let $M_\phi$ be a once-punctured torus bundle that is hyperbolic.
Then there are at most finitely many compact orientable 3-manifolds $M_1, M_2,
\dots , M_n$ so that $\wh{\pi_1(M_i)} \cong \wh{\pi_1(M_\phi)}$ and
these are all hyperbolic once-punctured torus bundles $M_{\phi_i}$
with $\tr(\phi)=\tr(\phi_i)$ and $\det\phi = \det\phi_i$
\end{theorem}

\begin{proof} This follows the proof of Theorem \ref{t:main}, except that 
now we have the finite ambiguity provided by Proposition \ref{p:unique} rather than
uniqueness.
\end{proof}

\begin{theorem}
\label{moregeeneral}
Let $K_1, K_2 \subset S^3$ be knots whose complements are hyperbolic with
$\Sh^3\smallsetminus K_i = {\H}^3/\Gamma_i$.
Assume that $\Sh^3\smallsetminus K_1$ fibres over the circle with fibre a
surface of genus $g$.  If $\widehat{\Gamma}_1\cong 
\widehat{\Gamma}_2$, then $\Sh^3\smallsetminus K_2$ is fibred with fibre a surface
of genus $g$.\end{theorem}

\begin{proof} From Theorem \ref{t:fibres} we know that $\Sh^3\smallsetminus K_2$ is fibred 
and from Lemma \ref{l:sameHat} we know that the fibres have the same euler characteristic. 
 That the genus is the same now follows, since in both cases the fibre is a Seifert surface
and so the surface has a single boundary component, hence
the same genus.
\end{proof}

\begin{theorem}
Let $M$ and $N$ be compact orientable  3-manifolds with $\wh{\pi_1(N)}\cong\wh{\pi_1(M)}$.
Suppose that $\partial M$ is an incompressible torus and that $b_1(M)=1$. If
$M$ fibres over the circle with fibre a surface of euler characteristic $\chi$,
then $N$ fibres over the circle with fibre a surface of euler characteristic $\chi$,
and $\partial N$ is a torus.
\end{theorem}

\begin{proof} This follows from Theorem \ref{t:fibres},  Lemma \ref{l:sameHat} and Corollary \ref{c:notClosed}.
\end{proof}
Theorem \ref{summary}(5) shows that the fundamental group
of a torus knot complement cannot have the same profinite completion as
that of a Kleinian group of finite co-volume.  

\begin{theorem}
\label{nottorus}
Let $K_1, K_2 \subset S^3$ be knots whose complements are geometric.
Let $\Gamma_i=\pi_1(S^3\smallsetminus K_i)$ and assume that 
$\widehat{\Gamma}_1\cong\widehat{\Gamma}_2$. Then:
\begin{itemize}
\item $K_1$ is hyperbolic if and only if $K_2$ is hyperbolic.
\item $K_1$ is fibred if and only if $K_2$ is fibred.
\end{itemize}
\end{theorem} 

The following is simply a summary of earlier results.

\begin{thmA}\label{t:compile} Let $N$ and $M$ be compact connected 3-manifolds. Assume
 $b_1(N)=1$ and $\pi_1(N)\cong F_r\rtimes_\phi\Z$. If $\wh{\pi_1(N)}\cong
\wh{\pi_1(M)}$, then $b_1(M)=1$ and
\begin{enumerate}
\item $\pi_1M\cong F_r\rtimes_\psi\Z$, for some $\psi\in{\rm{Out}}(F_r)$, 
\item $M$ is fibred, and
\item $\partial M$ is either a torus or a Klein bottle.
\end{enumerate}
Moreover, when $r=2$, 
\begin{enumerate}
\item[(4)] for each $N$ there are only finitely many
possibilities for $M$,  
\item[(5)] $M$ is hyperbolic if and only if $N$ is hyperbolic, and
\item[(6)] if $M$ is hyperbolic, it is a once-punctured torus bundle.
\end{enumerate}
\end{thmA}

\end{document}